\documentclass[11pt,leqno]{amsart}

\usepackage{hyperref}

\newtheorem{theorem}{Theorem}[section]
\newtheorem{lemma}[theorem]{Lemma}
\newtheorem{proposition}[theorem]{Proposition}
\newtheorem{corollary}[theorem]{Corollary}

\theoremstyle{definition}

\title{A note on automorphisms of finite $p$-groups}

\author[G. A. Fern\'{a}ndez-Alcober]{Gustavo A. Fern\'{a}ndez-Alcober}
\address{Gustavo A. Fern\'{a}ndez-Alcober: Department of Mathematics, University of the Basque Country UPV/EHU, 48080 Bilbao, Spain}
\email{gustavo.fernandez@ehu.eus}

\author[A. Thillaisundaram]{Anitha Thillaisundaram} 
\address{Anitha Thillaisundaram: Mathematisches Institut, 
Heinrich-Heine-\break Universit\"{a}t, 40225 D\"usseldorf, Germany}
\email{anitha.t@cantab.net}

\keywords{finite $p$-groups, automorphisms}
\subjclass[2010]{Primary  20D15;  Secondary 20D45}
\thanks{The first author is supported by by the Spanish Government, grants MTM2011-28229-C02-02 and MTM2014-53810-C2-2-P,
and by the Basque Government, grant IT753-13.
The second author, who is funded by the Alexander von Humboldt Foundation, would like to thank the University of the Basque Country for its hospitality.}

\date{26th October 2015}

\begin{document}
\begin{abstract} 
Let $G$ be a finite non-cyclic $p$-group of order at least $p^3$. If $G$ has an abelian maximal subgroup, or if $G$ has an elementary abelian centre with $C_G(Z(\Phi(G))) \ne \Phi(G)$, then $|G|$ divides $|\text{Aut}(G)|$.
\end{abstract}

\maketitle

\section{Introduction}

From the 1970s, the question `Does every finite non-cyclic $p$-group $G$ of order $|G|\ge p^3$ have $|G|$ dividing $|\text{Aut}(G)|$?' began to take form.
Observe that, if $\text{Out}(G)=\text{Aut}(G)/\text{Inn}(G)$ denotes the group of outer automorphisms of $G$, this is equivalent to asking whether $|\text{Out}(G)|$ divides $|Z(G)|$. 
Over the past fifty years, this question was partially answered in the affirmative for specific families of $p$-groups, for instance $p$-abelian $p$-groups, $p$-groups of class 2, $p$-groups of maximal class, etc (see \cite{me} for a fairly up-to-date list).
This led many to believe that the complete answer might be yes, which is why the question was reformulated as a conjecture: ``If $G$ is a finite non-cyclic $p$-group with $|G|\ge p^3$, then $|G|$ divides $|\text{Aut}(G)|$".

What is more, Eick \cite{Eick} proved that all but finitely many 2-groups of a fixed coclass satisfy the conjecture.
Couson \cite{Couson} generalized this to $p$-groups for odd primes, but only to infinitely many $p$-groups of a fixed coclass.
The coclass theory shed new light on the conjecture, and provided more evidence as to why it could be true. Looking at past efforts, it could also be said that an underlying theme was cohomology,
which hinted that the full conjecture might be settled using such means.

However, it came as a surprise that the conjecture is false. Very recently, Gonz\'{a}lez-S\'{a}nchez and Jaikin-Zapirain  \cite{Jon} disproved the conjecture using Lie methods, where the question was first translated
into one for Lie algebras. The main idea was to use the examples of Lie algebras with derivation algebra of smaller dimension, from which they constructed a family of examples of $p$-groups with small automorphism group. 
We remark that these counter-examples are powerful and $p$-central, which means that $G' \le G^p$ and $\Omega_1(G)\le Z(G)$ respectively, if $p$ is odd, and that $G'\le G^4$ and
$\Omega_2(G)\le Z(G)$ respectively, if $p=2$.

Now a new question may be formulated: which other finite non-cyclic $p$-groups $G$ with $|G|\ge p^3$ have $|G|$ dividing
$|\text{Aut}(G)|$?
In this short note, we prove that for $G$ a finite non-cyclic $p$-group with $|G|\ge p^3$, if $G$ has an abelian maximal subgroup, or if $G$ has elementary abelian centre and $C_G(Z(\Phi(G)))\ne \Phi(G)$, then $|G|$ divides $|\text{Aut}(G)|$.
The latter is a partial generalization of Gasch\"{u}tz' result \cite{Gaschuetz} that $|G|$ divides $|\text{Aut}(G)|$ when the centre has order $p$.

\vspace{10pt}

\noindent
\textit{Notation.\/}
We use standard notation in group theory.
All groups are assumed to be finite and $p$ always stands for a prime number.
For $M,N$ normal subgroups in $G$, we set $\text{Aut}_N^M(G)$ to be the subgroup of automorphisms of $G$ that centralize $G/M$ and $N$, and let $\text{Out}^M_N(G)$ be its corresponding image in $\text{Out}(G)$.
When $M$ or $N$ is $Z(G)$, we write just $Z$ for conciseness.
On the other hand, if $G$ is a finite $p$-group then $\Omega_1(G)$ denotes the subgroup generated by all elements of $G$ of order $p$.


\section{An abelian maximal subgroup}

Let $G$ be a finite $p$-group with an abelian maximal subgroup $A$. We collect here a few well-known results (see \cite[Lemma 4.6 and its proof]{Isaacs}).

\begin{theorem} \label{Theorem63}
 Let $G$ be a group having an abelian normal subgroup $A$ such that the quotient $G/A=\langle gA \rangle$, with $g\in G$, is cyclic. 

Then (i) $ G' = \{ [a,g] \mid a\in A \}$ and (ii) $|G'|=|A:A\cap Z(G)|$.

 \end{theorem}
 
 \begin{corollary} \label{Corollary64}
  Let $G$ be a finite non-abelian $p$-group having an abelian maximal subgroup. Then $|G:Z(G)|=p|G'|$.
 \end{corollary}
 
In \cite{key-Webb}, Webb uses the following approach to find non-inner
automorphisms of $p$-power order, which we will use in the forthcoming theorem. For a maximal subgroup $M$ of
$G$, we first define two homomorphisms on $Z(M)$.
Let $g\in G$ be such that $G/M=\langle gM\rangle$,
then\[
\,\tau:m\quad \mapsto\, \quad g^{-p}(gm)^{p}=m^{g^{p-1} + \ldots + g + 1} \]
\[
\quad\quad\quad \qquad \gamma:m\quad \mapsto \qquad\quad  [m,g]=m^{g-1}\qquad \qquad\qquad \qquad\,   \]
for all $m\in Z(M)$.

We have $\text{im }\gamma\subseteq\ker\tau$ and $\text{im }\tau\subseteq\ker\gamma=Z(G)\cap M$.

\begin{corollary} \label{Webb}
\cite{key-Webb} Let $G$ be a finite non-abelian $p$-group and $M$
a maximal subgroup of $G$ containing $Z(G)$. Then $G$ has a non-inner
automorphism of $p$-power order inducing the identity on $G/M$
and $M$ if and only if $\textup{im }\tau\neq\ker\gamma$.
\end{corollary}

We remark that the proof of the above also tells us
that if $\text{im }\tau\neq\ker\gamma$, then $|\textup{Out}_{M}^{M}(G)|=|\ker\gamma|/|\text{im }\tau|$.

 \begin{theorem} \label{AbelianMaximal}
  Let $G$ be a finite non-cyclic $p$-group with $|G|\ge p^3$ and with an abelian maximal subgroup $A$. Then $|G|$ divides $|\textup{Aut}(G)|$.
 \end{theorem}
 
 \begin{proof}
 We work with the subgroup $\text{Aut}^Z(G)$ of central automorphisms in $\text{Aut}(G)$. Now
 \[
  |\text{Aut}^Z(G) \text{Inn}(G)|=\frac{|\text{Aut}^Z(G)|\cdot |\text{Inn}(G)|}{|\text{Aut}^Z(G)\cap \text{Inn}(G)|} = \frac{|\text{Aut}^Z(G)|\cdot |G/Z(G)|}{|Z_2(G)/Z(G)|}
 \]
and hence it suffices to show that $|\text{Aut}^Z(G)|\ge |Z_2(G)|$. 

According to Otto \cite{Otto}, when $G$ is the direct product of an abelian $p$-group $H$ and a $p$-group $K$ having no abelian direct factor, then one has $|H|\cdot |\text{Aut}(K)|_p$ divides $|\text{Aut}(G)|$. Hence, we may assume that $G$ has no abelian direct factor. It then follows by Adney and Yen \cite{key-PNgroup} that $|\text{Aut}^Z(G)|=|\text{Hom}(G/G',Z(G))|$.
 
 By Corollary \ref{Corollary64}, we have $|G:Z(G)|=p|G'|$ and equally, 
\begin{equation}\label{Cor2.2}
|G:G'|=p|Z(G)|. 
\end{equation}

 Let $H=G/Z(G)$. Then $A/Z(G)$ is an abelian maximal subgroup of $H$. Applying Corollary \ref{Corollary64} to $H$ yields
 \[
  |H:H'|=p|Z(H)|,
 \]
 so
 \[
  |G:G'Z(G)|=p|Z_2(G):Z(G)|.
 \]
 Hence
 \begin{equation}\label{**}
  |Z_2(G)|=\frac{1}{p}\cdot |G'\cap Z(G)|\cdot |G:G'|.
 \end{equation}
 
Combining this with (\ref{Cor2.2}) gives
 \begin{equation} \label{3}
  |Z_2(G)|=|G' \cap Z(G)|\cdot |Z(G)|.
 \end{equation}
 
 Next we have $G'=\{[a,g] \mid a\in A \}$ by Theorem \ref{Theorem63}. By Webb's construction, we know that $\text{im}\, \gamma \subseteq \ker \tau$ and here $\text{im}\, \gamma =[A,g] = G'$. Now for $a\in A$ and $g$ as in Theorem \ref{Theorem63}, we have
 \[
  [a,g]^{g^{p-1}+\ldots +g+1} =1.
 \]
 
 We claim that $\exp(G' \cap Z(G))=p$. For, if $[a,g]\in Z(G)$, then $[a,g]^g=[a,g]$. Consequently,  
 \[
  1=[a,g]^{g^{p-1}+\ldots +g+1} = [a,g]^p
 \]
and thus $o([a,g])\le p$. 
 
 Clearly the minimal number $d:=d(G)$ of generators of $G$ is at least 2. In order to proceed, we divide into the following two cases: (a) $\exp (G/G')\ge \exp Z(G)$, and (b) $\exp (G/G') \le \exp Z(G)$.
 
 \textbf{Case (a):} Suppose $\exp(G/G')\ge \exp Z(G)$. We express $G/G'$ as
 \[
  G/G' = \langle g_1 G' \rangle \times \langle g_2 G' \rangle \times \ldots \times \langle g_d G' \rangle
 \]
 where $o(g_1 G')= \exp (G/G')$ and by assumption $d\ge 2$.
 
 We consider homomorphisms from $G/G'$ to $Z(G)$. The element $g_1 G'$ may be mapped to any element of $Z(G)$, and $g_2 G'$ may be mapped to any element of $G' \cap Z(G)$, which is of exponent $p$.
 
 Thus, with the aid of (\ref{3}), 
 \[
  |\text{Hom}(G/G',Z(G))|\ge |Z(G)|\cdot |G' \cap Z(G)| =|Z_2(G)|.
 \]
 
 \textbf{Case (b):} Suppose $\exp(G/G') \le \exp Z(G)$. Similarly we express $Z(G)$ as
 \[
  Z(G) = \langle z_1 \rangle \times \langle z_2 \rangle \times \ldots \times \langle z_r \rangle
 \]
 where $r=d(Z(G))$ and $o(z_1)=\exp Z(G)$.
 
 We consider two families of homomorphisms from $G/G'$ to $Z(G)$. First,
 \[
 \quad  G/G' \rightarrow Z(G)
 \]
 \[
  g_i G' \mapsto z_1^{b_i}
 \]
 for $1\le i\le d$, where $b_i$ is such that $o(z_1^{b_i})$ divides $o(g_i G')$. This gives rise to $|G/G'|$ homomorphisms.
 
 Next, we consider all homomorphisms from $G/G'$ to $Z(G)$ where each element $g_i G'$ is mapped to any element of order $p$ in $\langle z_2 \rangle \times \ldots \times \langle z_r \rangle$. 
 This gives
 \[
 (p^{r-1})^d\ge p^{r-1}=\frac{|\Omega_1(Z(G))|}{p} \ge \frac{|G' \cap Z(G)|}{p}
 \]
 different homomorphisms.
 
 Multiplying both together and then using (\ref{**}), we obtain \[|\text{Hom}(G/G',Z(G))|\ge \frac{|G/G'|\cdot |G' \cap Z(G)|}{p} = |Z_2(G)|.\]

 \end{proof}

\section{Elementary abelian centre}

Let $G$ be a finite $p$-group with elementary abelian centre.  In order to prove that $|G|$ divides $|\text{Aut}(G)|$, we may assume, upon consultation of \cite{Faudree} and the final remarks in \cite{Hummel}, that $Z(G)<\Phi(G)$. One of the following three cases exclusively occurs.

\textbf{Case 1.} $Z(M)=Z(G)$ for some maximal subgroup $M$ of $G$.

\textbf{Case 2.} $Z(M)\supset Z(G)$ for all maximal subgroups $M$ of $G$.
Then either 

\textbf{(A)} $C_{G}(Z(\Phi(G)))\neq\Phi(G)$; or 

\textbf{(B)} $C_{G}(Z(\Phi(G)))=\Phi(G)$.

$\ $

The main result of this section is to show that if $G$ is a finite $p$-group with elementary
abelian centre and not in Case 2B, then $|G|$ divides $|\text{Aut}(G)|$. 

With regard to Case 2B, we would also like to mention another long-standing conjecture for finite $p$-groups: does there always exist a non-inner automorphism of order $p$? The case 2B is the only remaining case for this conjecture (see  \cite{DS}).


First we deal with Case 1.
\begin{lemma}
\label{lem:1}\cite[Lemma 2.1(b)]{key-Muller} Suppose $M$ is a maximal subgroup
of $G$. If $Z(M)\subseteq Z(G)$ then $\textup{Aut}_{M}^{Z}(G)$ is
a non-trivial elementary abelian $p$-group such that $\textup{Aut}_{M}^{Z}(G)\cap\textup{Inn}(G)=1$.
\end{lemma}

We comment that the proof of this result in \cite{key-Muller} tells
us that \[
|\textup{Aut}_{M}^{Z}(G)|=|\textup{Hom}(G/M,Z(M))|=|\Omega_{1}(Z(M))|.\]

\begin{lemma}
\label{lem:Two}Let $G$ be a finite $p$-group with elementary abelian
centre. Suppose that $Z(M)=Z(G)$ for some maximal subgroup $M$ of
$G$. Then $|G|$ divides $|\textup{Aut}(G)|$.
\end{lemma}
\begin{proof}
Using Lemma \ref{lem:1} and the above comment, it follows that
\[|\text{Out}(G)|_{p}\ge|\text{Aut}_{M}^{Z}(G)|=|Z(G)|.\] Hence $|G|$
divides $|\text{Aut}(G)|$ as required.
\end{proof}

Next, suppose $G$ is as in Case 2(A). We will need the following.
\begin{theorem}
\cite{key-Webb} \label{thm:Webb}Let $G$ be a finite non-abelian
$p$-group. Then $p$ divides the order of $\textup{Out}_Z(G)$.
\end{theorem}

Now we present our result.
\begin{proposition}
\label{pro:One}Let $G$ be a finite $p$-group with elementary abelian
centre, such that $C_{G}(Z(\Phi(G)))\ne\Phi(G)$ and $Z(M)\supset Z(G)$ for all maximal subgroups $M$ of $G$. Then $|G|$ divides $|\textup{Aut}(G)|$.
\end{proposition}

\begin{proof}
By M\"{u}ller \cite[proof of Lemma 2.2]{key-Muller}, there exist maximal subgroups $M$ and $N$
such that $G=Z(M)N$ and $Z(G)=Z(M)\cap N$.
Write $G/M=\langle gM\rangle$ for some $g\in G$, and consider the homomorphism
\[
\begin{matrix}
\tau_{M} & \colon & Z(M) & \longrightarrow & Z(M)\hfill
\\
& & m & \longmapsto & g^{-p}(gm)^{p}=m^{g^{p-1}+ \cdots+g+1}.
\end{matrix}
\]
Since
\[
\frac{Z(M)}{Z(G)}
=
\frac{Z(M)}{Z(M)\cap N}
\cong
\frac{Z(M)N}{N}
=
\frac{G}{N}
\] 
and $Z(G)\subseteq \ker\tau_M$, it follows that $|\text{im }\tau_M|\le p$.
If $\text{im }\tau_{M}=1$, the remark after Corollary \ref{Webb} implies that $|\text{Out}_{M}^{M}(G)|=|Z(G)|$ and 
we are done. Hence we assume that $\text{im }\tau_{M}\cong C_{p}$.
If $Z(G)\cong C_{p}$, then $|G|$ divides $|\text{Aut}(G)|$ by Gasch\"{u}tz \cite{Gaschuetz}.
So we may assume that $|Z(G)|>p$. Again by the same remark, we
have $|\text{Out}_{M}^{M}(G)|=|Z(G)|/p$. 

By \cite[Lemma 2.2]{key-Muller}, it follows that $G$ is a central product of subgroups
$R$ and $S$, where $R/Z(R)\cong C_p \times C_p$ and $Z(R)=Z(G)=Z(S)=R\cap S$. Furthermore, $R=Z(M)Z(N)$ and $S=M\cap N=C_{G}(R)$. 

 By Theorem \ref{thm:Webb}, there exists a non-inner automorphism $\beta\in\text{Aut}_{Z}(S)$ of $p$-power order. As observed by M\"{u}ller \cite[Section 3]{key-Muller}, the automorphism $\beta$ extends uniquely to some non-inner  $\gamma\in\text{Aut}_{Z}(G)$ with trivial action on $R$.

Certainly $\gamma$ does not act trivially
on $M\cap N=S$, so $\gamma\not\in\text{Aut}_{M}(G)$. Let $\overline{\gamma}$ be the image of $\gamma$ in $\text{Out}(G)$. We now show that $\overline{\gamma}\not \in \text{Out}^M_M(G)$. On the contrary, suppose that $\gamma = \rho. \iota$ where $\rho \in \text{Aut}^M_M(G)$ and $\iota \in \text{Inn}(G)$. As $S\le M$, we have $\beta(s)=\gamma(s)=\rho (s)^x=s^x$ for all $s\in S$ and a fixed $x\in G$. Writing $x=rs'$ for some $r\in R$, $s'\in S$ and recalling that $S=C_G(R)$, we have $\beta(s)=s^{s'}$.
This implies that $\beta \in \text{Inn}(S)$, a contradiction.
Thus \[|\text{Out}(G)|_{p}\ge|\langle\overline{\gamma},\text{Out}_{M}^{M}(G)\rangle|\ge|Z(G)|.\]
It follows that $|G|$ divides $|\text{Aut}(G)|$.
\end{proof}

\thanks{The authors are grateful to the various people who helped to read and
improve this manuscript.}


\begin{thebibliography}{99}

\bibitem{key-PNgroup} J. E. Adney and T. Yen, Automorphisms of
a $p$-group, \emph{Ill. J. Math.} \textbf{9} (1965), 137--143.

\bibitem{Couson} M. Couson, \emph{On the character degrees and automorphism groups of finite $p$-groups by coclass}, PhD Thesis, Technische Universit\"{a}t Braunschweig, Germany, 2013.

\bibitem{DS} M. Deaconescu and G. Silberberg, Noninner automorphisms of order $p$ of finite $p$-groups, \emph{J. Algebra} \textbf{250} (2002), 283-287.

\bibitem{Eick} B. Eick, Automorphism groups of 2-groups, \emph{J. Algebra} \textbf{300} (1) (2006), 91--101.

\bibitem{Faudree} R. Faudree, A note on the automorphism group
of a $p$-group, \emph{Proc. Amer. Math. Soc.} \textbf{19} (1968),
1379--1382.

\bibitem{Gaschuetz} W. Gasch\"{u}tz, Nichtabelsche $p$-Gruppen
besitzen \"{a}ussere $p$-Automorphismen (German), \emph{J. Algebra}
\textbf{4} (1966), 1--2.

\bibitem{Jon} J. Gonz\'{a}lez-S\'{a}nchez and A. Jaikin-Zaipirain, Finite $p$-groups with small automorphism group, 
\emph{Forum Math. Sigma} \textbf{3} (2015), e7.

\bibitem{Hummel} K. G. Hummel, The order of the automorphism
group of a central product, \emph{Proc. Amer. Math. Soc.} \textbf{47}
(1975), 37--40.

\bibitem{Isaacs} I. M. Isaacs, \emph{Finite group theory}, Graduate Studies in Mathematics, Vol. 92, Amer. Math. Soc., Providence, 2008.

\bibitem{key-Muller} O. M\"{u}ller, On $p$-automorphisms of
finite $p$-groups, \emph{Arch. Math.} \textbf{32} (1979), 553--538.

\bibitem{Otto} A. D. Otto, Central automorphisms of a finite
$p$-group, \emph{Trans. Amer. Math. Soc.} \textbf{125} (1966), 280--287.

\bibitem{me} A. Thillaisundaram, The automorphism group for
$p$-central $p$-groups, \emph{Internat. J. Group Theory} \textbf{1}
(2) (2012), 59--71.

\bibitem{key-Webb} U. H. M. Webb, An elementary proof of Gasch\"{u}tz'
theorem, \emph{Arch. Math.} \textbf{35} (1980), 23--26.

\bibitem{key-2} M. K. Yadav, On automorphisms of finite $p$-groups,
\emph{J. Group Theory} \textbf{10} (6) (2007), 859--866.

\end{thebibliography}
\end{document}